\documentclass[12pt]{article}

\usepackage{amsmath}
\usepackage{amsfonts}
\usepackage{verbatim}
\usepackage{amsthm}
\usepackage{amssymb}
\usepackage{epsfig}
\usepackage{pdfsync}
\usepackage{float}

\newtheorem{theorem}{Theorem}[section]
\newtheorem{proposition}[theorem]{Proposition}

\theoremstyle{definition}

\newcommand{\mc}[0]{\mathcal}

\newcommand{\dual}[0]{^*}

\newcommand{\wh}[0]{\widehat}
\newcommand{\wt}[0]{\widetilde}

\begin{document}

\title{A Basic Structure for Grids in Surfaces}
\author{Lowell Abrams\footnote{Department of Mathematics, The George Washington
University, Washington, DC 20052. \underline{Email:} labrams@gwu.edu}\: and Daniel Slilaty\footnote{Department of
Mathematics and Statistics, Wright State University, Dayton, OH 45435. \underline{Email:} daniel.slilaty@wright.edu.
Work partially supported by a grant from the Simons Foundation \#246380.}} \maketitle

\abstract{A graph $G$ embedded in a surface $S$ is called an $S$-\emph{grid} when every facial boundary walk has length
four, that is, the topological dual graph of $G$ in $S$ is 4-regular. Aside from the case where $S$ is the torus or
Klein bottle, an $S$-grid must have vertices of degrees other than four. Let the sequence of degrees other than four in
$G$ be called the \emph{curvature sequence} of $G$. We give a succinct characterization of $S$-grids with nonempty
curvature sequence $L$ in terms of graphs that have degree sequence $L$ and are immersed in a certain way in $S$;
furthermore, the immersion associated with the $S$-grid $G$ is unique and so our characterization of $S$-grids also
partitions the collection of all $S$-grids.}

\section{Introduction}

The reader is expected to be familiar with the basics of topological graph theory as in Gross and Tucker
\cite{GrossTucker:Book}; all terminology that we do not define is from there.

Given a closed surface $S$, an $S$-\emph{grid} is an embedding of a graph $G$ in $S$ such that every facial boundary
walk has length four, that is, the topological dual graph of $G$ is 4-regular. An $S$-grid might alternatively be
termed a \emph{quadrangulation} of $S$; however, we will use the term ``grid" in this paper. This is probably the
weakest sort of definition for a ``quadrangulation" or ``grid"; other studies often place additional constraints on the
embedding.

Other than the case in which $S$ is the torus or Klein bottle, any $S$-grid must have vertices of degrees other than
four. A very explicit characterization of $S$-grids in the torus and Klein bottle with every vertex of degree four
(along with the additional property that the four faces around each vertex along with their boundaries form a $2\times
2$ square grid) was initially given by Thomassen \cite{Thomassen:Tilings}; a slightly different formulation is given by
M\'arquez, de Mier, Noy, Revuelta \cite{DeMierNoy:GridGraphs}.

If ``most" of the vertices of an $S$-grid are of degree four, then $G$ has ``large" areas that are annular or appear as
the standard, geometrically-flat, infinite $\{4,4\}$-planar lattice.

\begin{center}
\includegraphics[height=80pt]{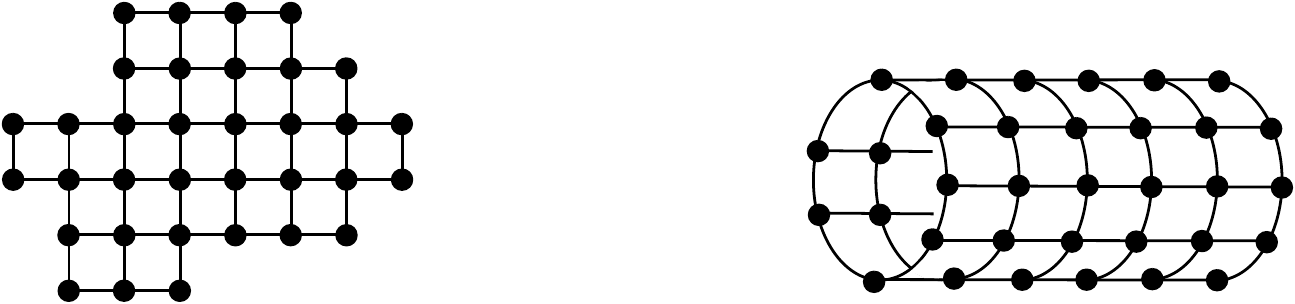}
\end{center}

\noindent In contrast to this, vertices that are not of degree four create the curvature necessary for an $S$-grid to
be finite when $S$ is not the torus or Klein bottle. As such, a vertex whose degree is not four is called a
\emph{curvature vertex}. Proposition \ref{P:PositiveNegativeCurvature} gives a relationship between the quantities and
degrees of curvature vertices in an $S$-grid.

\begin{proposition} \label{P:PositiveNegativeCurvature}
If $G$ is an $S$-grid with $v_i$ vertices of degree $i$ then, \[3v_1+2v_2+v_3=4\chi(S)+\sum_{i\geq5}(i-4)v_i.\]
Furthermore, if $\chi(S)\neq0$, then there are curvature vertices.\end{proposition}
\begin{proof}
If $G$ has $f$ faces and $e$ edges, then $\sum_i i v_i=2e$. Also, $4f=2e$ and $\big(\sum_i v_i\big)-e+f=\chi(S)$ which
when combined together yield $4(\sum_i v_i)=4\chi(S)+2e$. Now subtracting we obtain $\sum_i (4-i)v_i=4\chi(S)$ which
yields our desired results.
\end{proof}

Of course only certain combinations of quantities and degrees of curvature vertices are arithmetically possible. Given
an $S$-grid $G$ having some curvature vertices, the degree sequence of $G$ with the 4's removed is called the
\emph{curvature sequence} of $G$.

Given a graph $H$ and a surface $S$, a \emph{transverse immersion} of $H$ in $S$ is an immersion of $H$ in $S$ where
the only self intersections are transverse crossings of edge segments. Given a transverse immersion of $H$ in $S$, let
$\wh H$ be the graph embedding in $S$ obtained by placing a vertex at each transverse crossing of $H$ in $S$. We say
that the transverse immersion is \emph{quadrangular} when $\wh H$ is an $S$-grid.

In Section \ref{S:Construction} we will give a characterization of $S$-grids that also yields an equivalence relation
on the collection of all $S$-grids with curvature sequence $L$. The equivalence classes will be defined by quadrangular
transverse immersions in $S$ of graphs having degree sequence $L$. In Section \ref{S:ArithmeticConditions} we discuss a
simple arithmetic condition on graphs having transverse immersions in $S$. In Section \ref{S:Overlay} we motivate the
study of $S$-grids by providing an overview of two natural classes of $S$-grids arising from general embeddings of
graphs in surfaces.

\section{Construction} \label{S:Construction}

Consider a graph $G$ embedded in a surface $S$ and a vertex $v$ of degree 4 in $G$ with incident edges
$e_1,e_2,e_3,e_4$ in rotational order. Say that edges (or a single loop) $e_i$ and $e_{i+2}$ are \emph{transverse} with
respect to $v$. A \emph{transverse walk} in $G$ is a $uv$-walk in which neither $u$ nor $v$ have degree 4 (possibly
with $u=v$), each internal vertex in the walk has degree 4 in $G$, and pairs of successive edges along the walk are
transverse. Note that for any choice of vertex $u$ not of degree 4 and incident edge $e$ there is a unique transverse
walk starting at $u$ and containing $e$; furthermore, this walk is a trail, that is, no edge is ever used twice in the
walk. Also note that no two distinct transverse walks ever share an edge.

Now consider a closed walk $W=v_1,e_1,v_2,e_2,\ldots,v_n,e_n,v_1$ in $G$ in which each $v_i$ has degree 4 in $G$ and
pairs of successive edges (including the pair $e_n,e_1$) are transverse. Call such a walk a \emph{transverse circuit}.
Note that if $e$ is an edge in $G$ that is not contained in a transverse walk, then there is a unique transverse
circuit $W$ (up to choice of starting vertex and reversal) containing $e$ and $W$ is a trail. Furthermore, if $W_1\neq
W_2$ and $W_i$ is a transverse walk or transverse circuit, then $W_1$ and $W_2$ share no edge in common. These facts
yield Proposition \ref{P:TansverseDecomposition}.

\begin{proposition} \label{P:TansverseDecomposition}
If $G$ is graph embedded in a surface $S$, then the edges of $G$ partition in exactly one way into transverse walks and
transverse circuits.
\end{proposition}

Now consider a given $S$-grid $G$ with nonempty curvature sequence $L$. Let $G_L$ be the graph whose vertices are the
curvature vertices of $G$ with an edge between $u$ and $v$ in $G_L$ if and only if there is a transverse $uv$-walk in
$G$. So now the transverse walks in $G$ provide a unique transverse immersion of $G_L$ in $S$ as described in
Proposition \ref{P:SkeletonGraph}.

\begin{proposition} \label{P:SkeletonGraph}
If $G$ is an $S$-grid with nonempty curvature sequence $L$, then there is a graph $G_L$ on the curvature vertices of
$G$ such that: $G_L$ has degree sequence $L$, there is a unique transverse immersion of $G_L$ in $S$ such that $G$
contains a subdivision of $\wh G_L$ as a subgraph, and the subdivided edges of $G_L$ are the transverse walks in $G$.
\end{proposition}

We call the uniquely obtained embedded graph $\wh G_L$ of Proposition \ref{P:SkeletonGraph} and Theorem
\ref{P:TransverseImmersionIsQuadrangular} the \emph{skeleton grid} of the $S$-grid $G$.

\begin{theorem} \label{P:TransverseImmersionIsQuadrangular}
If $G$ is an $S$-grid with nonempty curvature sequence $L$, then the uniquely obtained embedded graph $\wh G_L$ of
Proposition \ref{P:SkeletonGraph} is an $S$-grid.
\end{theorem}
\begin{proof}
Let $\wt G_L$ be the subgraph of $G$ that is a subdivision of $\wh G_L$. Let $\mc R$ be the collection of regions into
which $\wt G_L$ subdivides $S$. It is not \emph{a priori} true that $R$ is a collection of 2-cells; however, we will
see that this is indeed the case. Consider some $R\in\mc R$ and let $G_R'$ be the subgraph of $G$ that is embedded in
$R$ including the boundary which is a closed walk in $\wt G_L$. Let $G_R$ be the surface obtained from $G_R'$ by
cutting along the boundary walk of $R$ in $\wt G_L$ so that the resulting boundary is a cycle. Let $v_c$ be the number
of times a copy of a vertex of $\wh G_L$ (i.e., a branch vertex of $\wt G_L$) appears on the boundary walk of $G_R$;
$v_s$ be the number of times a copy of a subdividing vertex appears on the boundary walk of $G_R$; $v_I$ be the number
of interior vertices of $G_R$; $e$ be the number of edges of $G_R$; $f$ the number of faces of $G_R$ (excluding the
outer face); and $l$ be the length of the boundary cycle of $G_R$. So now $4f=2e-l$ and $2e=2v_c+3v_s+4v_I$. Now
calculating the Euler characteristic of $G_R$ we obtain\begin{eqnarray*}
                                   \chi(G_R) &=& v_c+v_s+v_I-e+f\\
                                    &=&  v_c+v_s+v_I-e+\textstyle\frac12e-\frac14l\\
                                    &=&  v_c+v_s+v_I-\textstyle\frac12e-\frac14(v_c+v_s)\\
                                    &=& \textstyle\frac34(v_c+v_s)+v_I-\frac12v_c-\frac34v_s-v_I\\
                                    &=& \textstyle\frac14v_c
                                 \end{eqnarray*} Of course, $\chi(G_R)$ is an integer and
$\chi(G_R)\leq 1$. Also $v_c>0$ because $G_R$ is defined by a region of $\wh G_L$. Thus $0<\frac14v_c=\chi(G_R)\leq1$
which implies that $v_c=4$ and that $G_R$ is a disk. Our result follows.
\end{proof}

So now, given an $S$-grid $G$ and its skeleton grid $\wh G_L$, again let $\wt G_L$ be the subdivision of $\wh G_L$ that
is a subgraph of $G$. Let $Q$ be a quadrilateral face of $\wh G_L$ and let $Q'$ be the corresponding face of $\wt G_L$.
We claim that the part of $G$ inside of $Q'$ is obtained as follows: subdivide opposite edges on the boundary of $Q$ an
equal number of times and then patch with a rectangular grid as shown in Figure \ref{F:Patching}.

\begin{figure}[H]
\begin{center}
\includegraphics[scale=1,page=1]{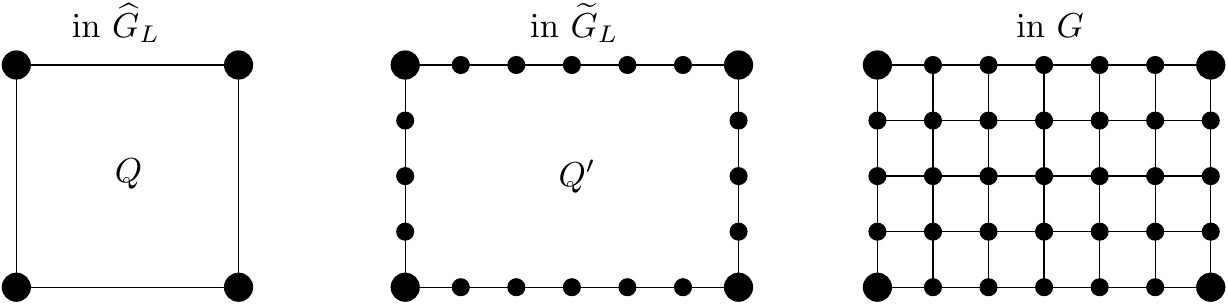}
\end{center}
\caption{Patching} \label{F:Patching}
\end{figure}

Showing that the part of $G$ inside of $Q$ is obtained in this fashion is easily done by the following inductive
argument. Let $e_1,e_2,e_3,e_4$ be the boundary walk of $Q$. Consider the edge $e_1$ and say that $e_1$ is subdivided
$t$ times in going from $Q$ to $Q'$. Each of these subdividing vertices on $e_1$ has exactly one incident edge in the
interior of $Q$. The only way in which quadrilateral faces may now be closed off is with a path of edges from $e_2$ to
$e_4$ (see Figure \ref{F:PatchingInduction}). Continuing by induction yields the desired structure.

\begin{figure}[H]
\begin{center}
\includegraphics[scale=0.8,page=10]{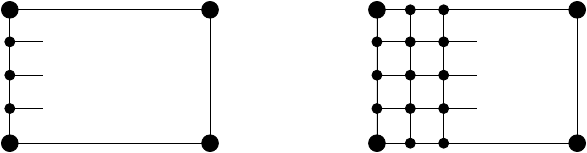}
\end{center}
\caption{Induction} \label{F:PatchingInduction}
\end{figure}

The only question remaining is what are the possible choices for the number of subdivisions for each edge. The
topological dual graph $(\wh G_L)\dual$ is 4-regular and so its edges partition into transverse circuits. As stated
before, the number of subdivisions for opposing sides of a quadrilateral face $Q$ must be the same. Hence each of the
edges that form a transverse circuit of $(\wh G_L)\dual$ must be subdivided the same number of times as the others.

We now have the following general construction method for any $S$-grid. Furthermore, this method partitions the class
of all $S$-grids into equivalence classes represented by their skeleton graphs. After choosing the graph $G_L$ in Step
1, it is not at all clear as to whether or not $G_L$ has a quadrangular transverse immersion in any closed surface $S$.
Thus most of the detail of $S$-grids is contained in Step 2 because Steps 3 and 4 can always be carried out
unambiguously after the completion of Step 2.

\begin{itemize}
\item[(1)] Take a graph $G_L$ without vertices of degree 4.
\item[(2)] Take a quadrangular transverse immersion of $G_L$ in a closed surface $S$ and its associated skeleton
    graph $\wh G_L$.
\item[(3)] Calculate the transverse circuits of $(\wh G_L)\dual$ and choose a non-negative integer $n_C$ for each
    transverse circuit $C$.
\item[(4)] Subdivide the edges $\wh G_L$ corresponding to $C$ $n_C$ times each and patch the resulting faces.
\end{itemize}

As an example of this construction consider the Wagner Graph $V_8$. Two distinct transverse immersions of $V_8$ in the
sphere are shown in Figure \ref{F:WagnerGraph}. The immersion on the right is quadrangular but the one on the left is
not. Let $\wh V_8$ be the skeleton graph obtained by the quadrangular transverse immersion.

\begin{figure}[H]
\begin{center}
\includegraphics[scale=.5,page=2]{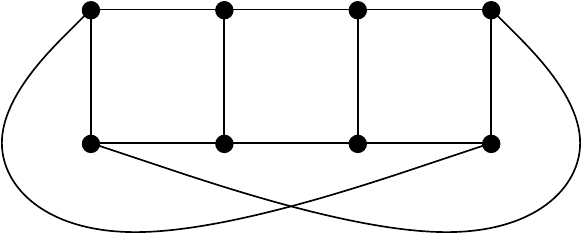}\hspace{1cm}
\includegraphics[scale=.5,page=3]{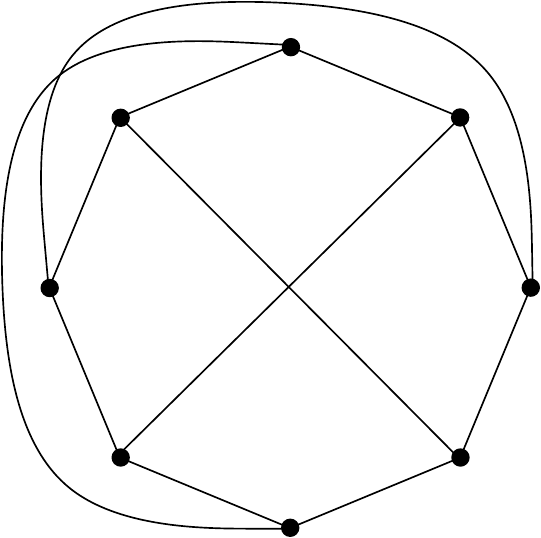}
\end{center}
\caption{Two transverse immersions of the Wagner Graph in the sphere. The one on the right is quadrangular.} \label{F:WagnerGraph}
\end{figure}

The edges of the topological dual graph of the spherical grid $\wh V_8$ form a single transverse circuit. Thus each
edge of $\wh V_8$ must be subdivided the same number of times and then each face is patched. In Figure
\ref{F:WagnerGraph2}, each edge is subdivided twice.

\begin{figure}[H]
\begin{center}
\includegraphics[scale=.5,page=4]{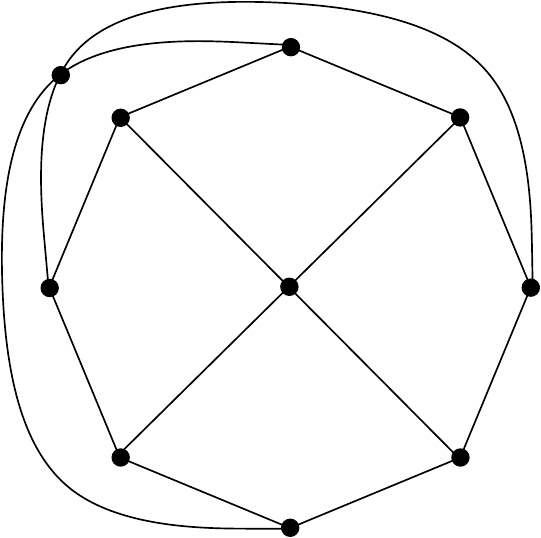}\hspace{1cm}
\includegraphics[scale=.4,page=5]{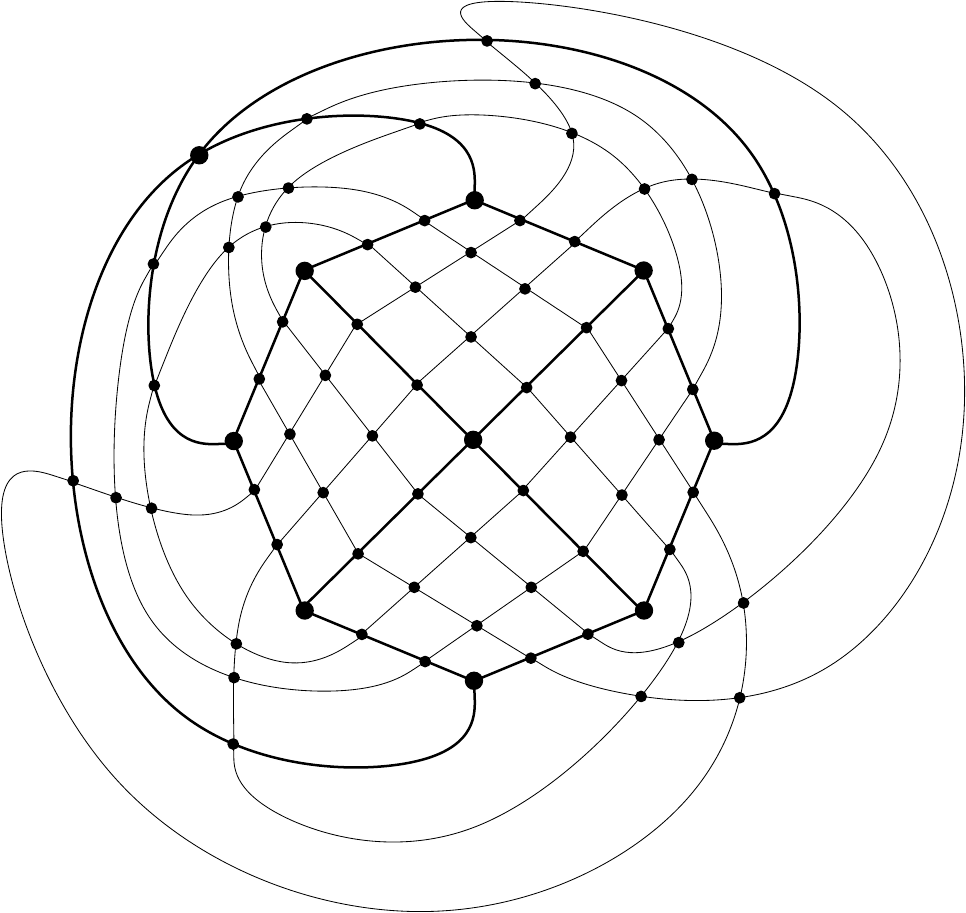}
\end{center}
\caption{A skeleton grid coming from a transverse immersion of the Wagner Graph
and a spherical grid obtained by subdividing each edge twice and then patching.} \label{F:WagnerGraph2}
\end{figure}

\section{Arithmetic Conditions} \label{S:ArithmeticConditions}

Given a graph $G$ without degree-4 vertices and a surface $S$, Proposition \ref{Arithmetic1} provides an arithmetic
condition that is necessary for $G$ to have a quadrangular transverse immersion in a given closed surface $S$.

\begin{proposition} \label{Arithmetic1}
If $G$ has a quadrangular transverse immersion in $S$, then $\chi(S)=|V(G)|-\frac12|E(G)|$. In particular, a given
graph $G$ without vertices of degree 4 can have a quadrangular transverse immersion in surfaces of only one possible
Euler characteristic.
\end{proposition}
\begin{proof}
Let $\wh G$ be the skeleton grid of transverse immersion of $G$ in $S$ with $v_4$ being the number of transverse
crossings used. Thus $|V(\wh G)|=|V(G)|+v_4$, $|E(\wh G)|=|E(G)|+2v_4$, and $f$ is the number of faces of the embedding
of $\wh G$, then $4f=2(|E(G)|+2v_4)$. We now have that\begin{eqnarray*}
                            \chi(S) &=& |V(G)|+v_4-(|E(G)|+2v_4)+\textstyle\frac12(|E(G)|+2v_4) \\
                             &=& |V(G)|-\textstyle\frac12|E(G)|
                          \end{eqnarray*}

\end{proof}

If a graph $G$ without degree-4 vertices does have a quadrangular immersion in a closed surface $S$, then even though
$G$ satisfies $\chi(S)=|V(G)|-\frac12|E(G)|$, two different quadrangular transverse immersions of $G$ may have
different numbers of transverse crossings. Figure \ref{F:AlternatingWheel} shows two quadrangular immersions of the
alternating 10-wheel with zero and five transverse crossings, respectively. Clearly this example generalizes to the
alternating $(4k+2)$-wheel for any $k\geq2$.

\begin{figure}[H]
\begin{center}
\includegraphics[scale=.35,page=6]{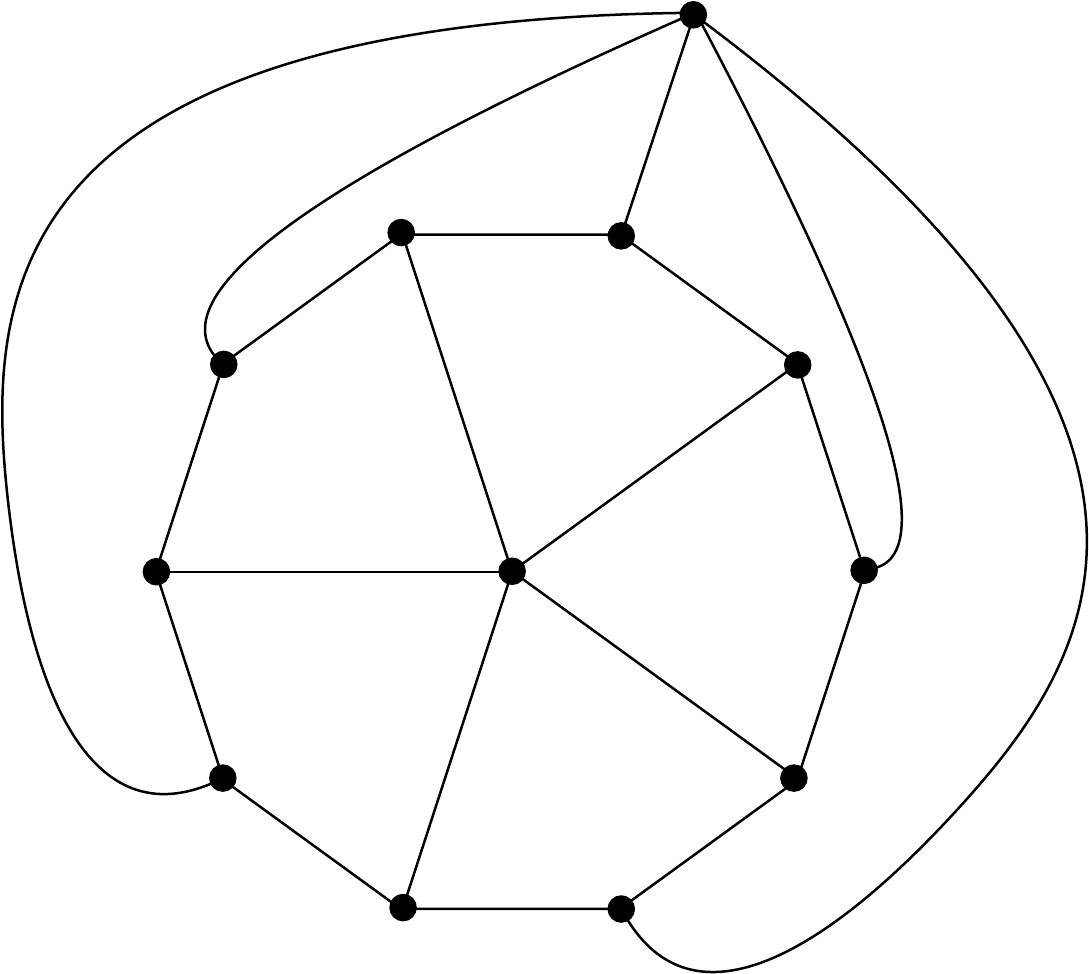}\hspace{1cm}
\includegraphics[scale=.4,page=7]{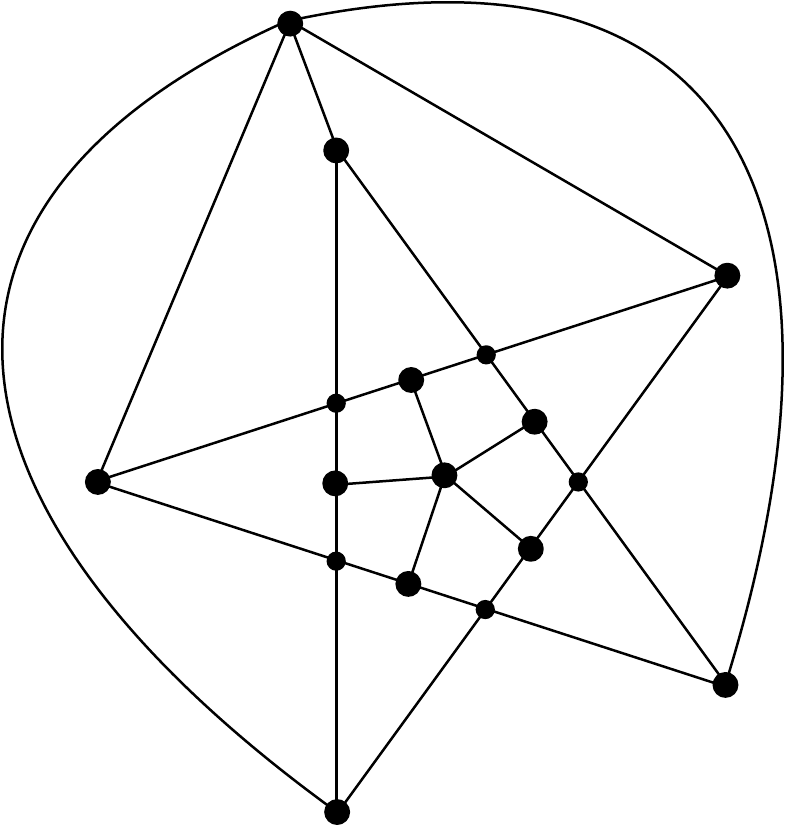}
\end{center}
\caption{Two different quadrangular transverse immersions of the same graph
with different numbers of transverse crossings.} \label{F:AlternatingWheel}
\end{figure}

\noindent In fact, in general it is not even possible to place an upper bound on the number of transverse crossings
(see Figure \ref{F:QuotientOfGrid}). Interestingly, the graph in Figure \ref{F:QuotientOfGrid} is the quotient of the
alternating $(4k+2)$-wheel under its $(2k+1)$-fold rotational symmetry.

\begin{figure}[H]
\begin{center}
\includegraphics[scale=.8,page=8]{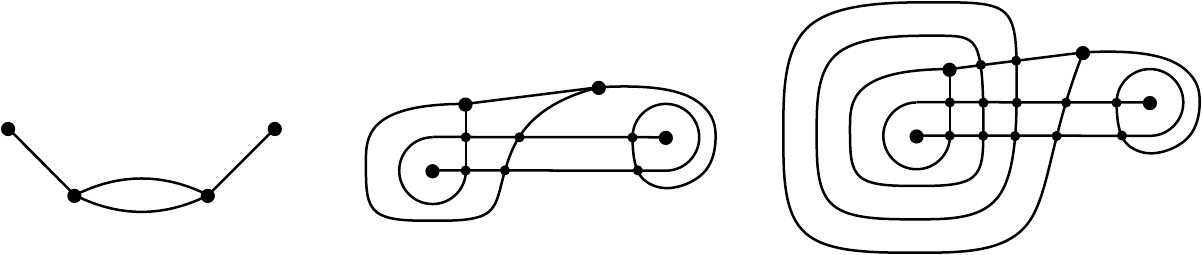}
\end{center}
\caption{Unbounded numbers of transverse crossings for the same graph.} \label{F:QuotientOfGrid}
\end{figure}

\section{Grids coming from arbitrary embeddings} \label{S:Overlay}

For any graph $H$ that is cellularly embedded in a closed surface $S$, there are two $S$-grids that are naturally
associated with the embedding of $H$ and its topological dual graph $H\dual$. These two types of $S$-grids also form
fundamental subclasses within the class of all $S$-grids. As such, $S$-grids are actually fundamental objects in
topological graph theory. In this section we give a short review of these $S$-grids.

\subsection{Radial Graphs}

The well-known \emph{radial graph}, $\mc R(H,H\dual)$ has vertex set $V(H)\cup V(H\dual)$. To describe the edges of
$\mc R(H,H\dual)$ consider a face $f$ of the embedding of $H$ in $S$ and its boundary walk
$v_1,e_1,v_2,e_2,\ldots,v_m,e_m,v_1$. The vertex $f\dual\in V(H\dual)$ has edges $g_1,\ldots,g_m$ connecting
respectively to $v_1,\ldots,v_m$. The radial graph is clearly an $S$-grid that is also bipartite with partite sets
$V(H)$ and $V(H\dual)$. The radial graph satisfies $\mc R(H,H\dual)=\mc R(H\dual,H)$ and the diagonals of the
quadrilateral faces of $\mc R(H,H\dual)$ connecting the vertices of $V(H)$ form $E(H)$ and the diagonals connecting the
vertices of $V(H\dual)$ form $E(H\dual)$. This latter observation yields Proposition \ref{P:RadialGraphsAreGrids}.

\begin{proposition}[Pisanski and Malni\v c \cite{Pisanski:MedialGraph}] \label{P:RadialGraphsAreGrids}
An $S$-grid $G$ is of the form $\mc R(H,H\dual)$ for some $H$ embedded in $S$ if and only if $G$ is bipartite.
\end{proposition}

Self-dual embeddings are nicely encoded by the radial graph in that the embeddings of $H$ and $H\dual$ are map
isomorphic if and only if $\mc R(H,H\dual)$ has a cellular automorphism that switches the partite sets $V(H)$ and
$V(H\dual)$. Self-dual embeddings have been studied from this viewpoint by Archdeacon and Richter
\cite{ArchdeaconRichter:PlanarSelfDual}, Archdeacon and Negami \cite{ArchdeaconNegami:ProjectivePlane}, and Abrams and
Slilaty \cite{AbramsSlilaty:Cellular}.

The topological dual graph of the radial graph $\mc R(H,H\dual)$ is known as the \emph{medial graph} $\mc M(H,H\dual)$.
The medial graph has been used by Archdeacon \cite{Archdeacon:MedialGraph} to give a unified presentation of the
concepts of voltage-graph and current-graph covering constructions. In \cite{MoffattMonaghan:TwistedDuality}, Moffatt
and Ellis-Monaghan describe the impressive result that all possible embeddings of the medial graph $\mc M(H,H\dual)$ in
all possible surfaces correspond precisely to the various notions of duality that generalize topological duality,
Petrie duality, and their partial versions and associated group actions.

\subsection{Overlay Graphs}

Consider a connected graph $H$ cellularly embedded in a closed surface $S$; its topological dual graph $H\dual$ is
therefore well defined, connected, and cellularly embedded. Say that all of $H$ (both vertices and edges) is colored
``red" and all of $H\dual$ is colored ``blue". Embed $H$ and $H\dual$ simultaneously in $S$ and at each edge/dual-edge
crossing point create a new vertex of degree four (which now has alternating red and blue edges in rotation around the
vertex) and say that this new vertex is ``white". The graph obtained is called the \emph{overlay graph} $\mc
O(H,H\dual)$. Certainly the overlay graph is an $S$-grid that is also bipartite with partite sets $\mathsf{Red\cup
Blue}$ and $\mathsf{White}$. Since the edge/dual-edge pairs of $H$ and $H\dual$ are the diagonals of the faces of the
radial graph $\mc R(H,H\dual)$ we also get that $\mc O(H,H\dual)$ is the radial graph of the radial graph of $H$ and
$H\dual$, that is, $\mc O(H,H\dual)=\mc R(\mc R(H,H\dual),\mc M(H,H\dual))$; recall that $\mc M(H,H\dual)$ is the
topological dual graph of $\mc R(H,H\dual)$.

The embedding of $H$ is self dual if and only if $\mc O(H,H\dual)$ has a cellular automorphism that reverses red and
blue colors and preserves white. The overlay graph was used by Servatius and Servatius
\cite{ServatiusServatius:Self-DaulCatalogueSphere,ServatiusServatius24,ServatiusServatius:Self-dualgraphs} to classify
self-dual embeddings in the sphere along with the pairing of their groups of color-preserving cellular automorphisms of
$\mc O(H,H\dual)$ as an index-2 subgroup of the group of red-blue switching cellular automorphisms of $\mc
O(H,H\dual)$. Graver and Hartung \cite{GraverHartung} do the same but with more detailed results for the special case
of self-dual embeddings of graphs having four trivalent vertices and the remaining vertices all of degree four.

For any closed surface $S$,  $\mc O(H,H\dual)$ is an $S$-grid that is bipartite and with the additional property that
all white vertices have degree four. Conversely, however, even if $G$ is a bipartite $S$-grid in which all white
vertices have degree four, it is not necessarily true that $G$ is of the form $\mc O(H,H\dual)$ for some $H$. An
additional condition that does ensure that $G$ has the form $\mc O(H,H\dual)$ is as follows: let $R(G)$ be the graph
obtained from $G$ by placing a diagonal edge connecting the black corners of each face and then deleting the white
vertices of $G$.

\begin{proposition} \label{P:GridOverlayEquivalence}
If $G$ is an $S$-grid, then $G=\mc O(H,H\dual)$ for some $H$ if and only if $G$ is bipartite, every white vertex of $G$
has degree 4, and $R(G)$ is bipartite.
\end{proposition}
\begin{proof}
The one direction is trivial. For the other direction, the fact that $R(G)$ is bipartite allows us to properly 2-color
(red and blue) the vertices of $R(G)$, which shows $G$ is of the form $\mc O(H,H\dual)$, as required.
\end{proof}

\bibliographystyle{amsplain}
\bibliography{SphericalGridsGeneral_bibfile}
\end{document}